\documentclass[12pt,leqno]{amsart}
\usepackage{enumerate}
\usepackage{amsrefs}
\usepackage{amsfonts, amsmath, amssymb, amscd, amsthm, bm, cancel}
\usepackage{url}
\usepackage{graphicx}
\usepackage[
linktocpage=true,colorlinks,citecolor=magenta,linkcolor=blue,urlcolor=magenta]{hyperref}
\usepackage{multicol}
\usepackage{comment}
\usepackage[margin=1in]{geometry}
\makeatletter
\@namedef{subjclassname@2020}{\textup{2020} Mathematics Subject Classification}
\makeatother
\theoremstyle{plain} \newtheorem{thm}{Theorem}[section]
\theoremstyle{plain} \newtheorem{prop}[thm]{Proposition}
\theoremstyle{plain} 
\theoremstyle{definition} 
\theoremstyle{plain} \newtheorem{cor}[thm]{Corollary}
\theoremstyle{plain} \newtheorem{lem}[thm]{Lemma}
\theoremstyle{plain} 
\theoremstyle{remark} \newtheorem{rmk}[thm]{Remark}
\theoremstyle{plain} \newtheorem{conj}{Conjecture}
 
\theoremstyle{definition}
\newtheorem*{ack}{Acknowledgments}
\theoremstyle{plain}

 \numberwithin{equation}{section}
 \allowdisplaybreaks

\newcommand{\bR}{\mathbb{R}}

\newcommand{\bM}{\mathbb{M}}

\DeclareMathOperator{\supp}{supp}

\DeclareMathOperator{\divg}{div}

\newcommand{\cH}{\mathcal{H}}

\newcommand{\cR}{\mathcal{R}}
\newcommand{\bH}{\mathbb{H}}

\newcommand{\bN}{\mathbb{N}}

\newcommand{\la}{\langle}
\newcommand{\ra}{\rangle}

\begin{document}
\title[Isodiametric bounds]{Isodiametric-type upper bounds for Steklov eigenvalues}

\author{Pingxin Gu}
\address{Yau Mathematical Sciences Center, Tsinghua University, Beijing 100084, P.R. China}
\email{\href{mailto:gupingxin@tsinghua.edu.cn}{gupingxin@tsinghua.edu.cn}}

\author{Yao Wan}
\address{Department of Mathematics, The Chinese University of Hong Kong, Shatin, N.T., Hong Kong}
\email{\href{mailto:yaowan@cuhk.edu.hk}{yaowan@cuhk.edu.hk}}




\begin{abstract}
Let $n\geq 2$. We establish isodiametric-type upper bounds for Steklov eigenvalues of bounded Lipschitz domains $\Omega$ in Euclidean space $\bR^n$ and hyperbolic space $\bH^n$. In Euclidean space, we prove
\[
D(\Omega)\sigma_k(\Omega)\leq C(n)k,
\]
with optimal dependence on $k$. For domains in either $\bR^n$ or $\bH^n$, we establish
\[
D(\Omega)\sigma_k(\Omega)\leq C(n)k^2\left(\frac{|\Omega|^{\frac{1}{n}}}{D(\Omega)}\right)^{\frac{n}{n-1}}.
\]
The exponent of the volume-diameter ratio is optimal in both settings, and with this exponent fixed, the quadratic dependence on $k$ is also optimal. Interpolating the preceding volume-diameter estimate with a volume-normalized Steklov bound, we further obtain
\[
D(\Omega)\sigma_k(\Omega)\leq C(n)k^{1+\frac{1}{2n-1}}
\]
in hyperbolic space. These results answer Open Question 4.37 affirmatively and resolve the hyperbolic part of Open Question 4.35, both posed in \cite{CGGS24}.
\end{abstract}

\maketitle

\section{Introduction}\label{sec-1}

The Steklov eigenvalue problem was introduced by V. A. Steklov \cite{Ste02} in 1902 during his study of heat conduction. A central problem in spectral geometry is to estimate Steklov eigenvalues in terms of geometric quantities. We refer the reader to the comprehensive surveys \cites{GP17,CGGS24}. In this paper, we study isodiametric-type upper bounds for Steklov eigenvalues in Euclidean and hyperbolic spaces.

Throughout the paper, we assume that $n\geq 2$ and let $\bM^n$ denote either $\bR^n$ or $\bH^n$, equipped with its standard metric. Let $\Omega\subset \bM^n$ be a bounded Lipschitz domain, and write $|\Omega|$ and $D(\Omega)$ for its Riemannian volume and diameter, respectively. The Steklov eigenvalue problem on $\Omega$ is given by
\begin{align}\label{eq-steklov}
\begin{cases}
\Delta u=0,&\text{in }\Omega,\\
\partial_{\nu}u=\sigma u,&\text{on }\partial\Omega,
\end{cases}
\end{align}
where $\Delta$ is the Laplace-Beltrami operator on $\Omega$ and $\nu$ is the outward unit normal to $\partial\Omega$. The Steklov problem \eqref{eq-steklov} has a discrete spectrum, listed with multiplicities,
\[
0=\sigma_0(\Omega)<\sigma_1(\Omega)\leq \sigma_2(\Omega)\leq \cdots\nearrow \infty.
\]

In \cite{BBG17}, Bogosel, Bucur, and Giacomini obtained the following isodiametric upper bound:
\begin{thm}[{\cite[Proposition 4.3]{BBG17}}]\label{thm-BBG17-4.3}
Let $\Omega\subset \bR^n$ be a bounded Lipschitz domain. Then there exists a constant $C(n)>0$ such that for every $k\in \bN$,
\begin{align}\label{eq-CGGS24-4.31}
D(\Omega)\sigma_k(\Omega)\leq C(n)k^{1+\frac{2}{n}}.
\end{align}
\end{thm}

After normalizing the volume, their proof splits into two cases. In the case of large diameter, the key ingredient is a delicate annular decomposition, which produces disjoint test functions and already yields a bound linear in $k$. In the complementary case, they combine the classical isoperimetric inequality with the spectral estimate of Colbois, El Soufi, and Girouard \cite[Theorem 1.2]{CEG11}. This is where the additional factor $k^{\frac{2}{n}}$ arises.

In dimension $2$, however, we have the basic fact $2D(\Omega)\leq \cH^1(\partial\Omega)$. Combining this with the sharp universal planar estimate $\cH^1(\partial\Omega)\sigma_k(\Omega)\leq 8\pi k$ of Girouard, Karpukhin, and Lagac\'e \cite{GKL21} (see also \cite[Theorem 3.16]{CGGS24}) yields
\[
D(\Omega)\sigma_k(\Omega)\leq 4\pi k.
\]
This raises the question of whether a linear bound in $k$ continues to hold in all dimensions. Along the same lines, the survey \cite{CGGS24} notes that
\[
\textit{It would be interesting to investigate the optimality of the power of $k$.}
\]
We answer this question by combining the annular decomposition with translations along Euclidean level hyperplanes. Our first main result is the following.
\begin{thm}\label{thm-main1}
Let $\Omega\subset \bR^n$ be a bounded Lipschitz domain. Then there exists a constant $C(n)>0$ such that for every $k\in \bN$,
\begin{align}\label{eq-main1}
D(\Omega)\sigma_k(\Omega)\leq C(n)k.
\end{align}
Moreover, the exponent $1$ of $k$ is optimal.
\end{thm}

A related volume-diameter bound was obtained by Al Sayed, Bogosel, Henrot, and Nacry \cite{ABHN21} under the additional assumption of convexity.
\begin{thm}[{\cite[Proposition 2.2]{ABHN21}}]\label{thm-ABHN21-2.2}
Let $\Omega\subset \bR^n$ be a bounded convex domain. Then for every $k\in \bN$, there exists a constant $C(n,k)>0$ such that
\[
D(\Omega)\sigma_k(\Omega)\leq C(n,k)\left(\frac{|\Omega|^{\frac{1}{n}}}{D(\Omega)}\right)^{\frac{n}{n-1}}.
\]
\end{thm}

Thus, for any family of convex domains with fixed diameter, $\sigma_k(\Omega_{\epsilon})\to 0$ for every fixed $k\in \bN$ whenever $|\Omega_{\epsilon}|\to 0$. Convexity is essential to the cone-comparison argument used in \cite{ABHN21}. It is therefore natural to ask whether the convexity assumption can be removed. This is precisely Open Question 4.37 in the survey \cite{CGGS24}:
\begin{conj}[{\cite[Open Question 4.37]{CGGS24}}]\label{conj-CGGS24-4.37}
Let $\Omega_{\epsilon}\subset \bR^n$ be a family of bounded Lipschitz domains with fixed diameter. If the volume of $\Omega_{\epsilon}$ tends to $0$ as $\epsilon\to 0$, can one say that all the eigenvalues of $\Omega_{\epsilon}$ tend to $0$?
\end{conj}

Using our method based on level hypersurfaces, we remove the convexity assumption and obtain the following estimate in both $\bR^n$ and $\bH^n$.
\begin{thm}\label{thm-main2}
Let $\Omega\subset \bM^n$ be a bounded Lipschitz domain. Then there exists a constant $C(n)>0$ such that for every $k\in \bN$,
\begin{align}\label{eq-main2}
D(\Omega)\sigma_k(\Omega)\leq C(n)k^2\left(\frac{|\Omega|^{\frac{1}{n}}}{D(\Omega)}\right)^{\frac{n}{n-1}}.
\end{align}
Furthermore, in both settings, the exponent $\frac{n}{n-1}$ of the volume-diameter ratio is optimal; with this exponent fixed, the quadratic dependence on $k$ is also optimal.
\end{thm}

As a corollary, the Euclidean case of Theorem \ref{thm-main2} gives an affirmative answer to Conjecture \ref{conj-CGGS24-4.37}.
\begin{cor}
Conjecture \ref{conj-CGGS24-4.37} is true.
\end{cor}

The isoperimetric spectral estimate of Colbois, El Soufi, and Girouard \cite{CEG11} is available in both Euclidean and hyperbolic spaces, as well as for domains contained in a hemisphere. This motivates the following question posed in \cite{CGGS24}.
\begin{conj}[{\cite[Open Question 4.35]{CGGS24}}]\label{conj-CGGS24-4.35}
Is it possible to get inequalities similar to those in Theorem \ref{thm-BBG17-4.3} for domains in the hyperbolic space or the sphere?
\end{conj}

For domains in hyperbolic space, we prove the following isodiametric bound.
\begin{thm}\label{thm-main0}
Let $\Omega\subset \bH^n$ be a bounded Lipschitz domain. Then there exists a constant $C(n)>0$ such that for every $k\in \bN$,
\begin{align}\label{eq-main0}
D(\Omega)\sigma_k(\Omega)\leq C(n)k^{1+\frac{1}{2n-1}}.
\end{align}
\end{thm}
The exponent $1+\frac{1}{2n-1}$ lies strictly between the optimal Euclidean exponent $1$ in \eqref{eq-main1} and the exponent $1+\frac{2}{n}$ in \eqref{eq-CGGS24-4.31}. It is obtained by interpolating Theorem \ref{thm-main2} with the volume-normalized estimate that follows from Theorem \ref{thm-CEG11-1.2} and the hyperbolic isoperimetric inequality. We do not know whether this exponent is optimal; the examples in Section \ref{sec-3} only show that the exponent of $k$ in any such uniform bound must be at least $1$. As a corollary, Theorem \ref{thm-main0} gives an affirmative answer to the hyperbolic part of Conjecture \ref{conj-CGGS24-4.35}.
\begin{cor}
Conjecture \ref{conj-CGGS24-4.35} is true in hyperbolic space.
\end{cor}

The remainder of this paper is organized as follows. In Section \ref{sec-2}, we collect basic properties of Steklov eigenvalues and the relevant isoperimetric inequalities. In Section \ref{sec-3}, we derive necessary restrictions on the exponents in a general volume-diameter estimate. In Section \ref{sec-4}, we combine the annular decomposition with translations along level hyperplanes to prove Theorem \ref{thm-main1}. Finally, in Section \ref{sec-5}, we introduce level hypersurfaces and use a one-dimensional reduction to prove Theorem \ref{thm-main2}; we then derive Theorem \ref{thm-main0} by interpolation.

\begin{ack}
    The authors are deeply grateful to Prof. Haizhong Li for drawing attention to literature relevant to this work and for his continued interest and encouragement.
	The first author is supported by the Shuimu Tsinghua Scholar Program.
	The second author is supported by Hong Kong RGC grants No. 14300623, No. 14304225, and No. 14308325.
\end{ack}

\section{Preliminaries}\label{sec-2}

Throughout this paper, for nonnegative quantities $A$ and $B$, we write $A\lesssim_nB$ if $A\leq C(n)B$ for some constant $C(n)>0$ depending only on $n$. The value of $C(n)$ may change from line to line. We write $A\gtrsim_n B$ if $B\lesssim_nA$, and $A\asymp_nB$ if both inequalities hold.

Let $\Omega\subset \bM^n$ be a bounded Lipschitz domain. The diameter $D(\Omega)$ of $\Omega$ is defined by
\[
D(\Omega)=\sup\{d_g(x,y):x,y\in \Omega\}.
\]

\subsection{Variational properties}
For $u\in H^1(\Omega)$ with nonzero boundary trace, define its Rayleigh-Steklov quotient by
\begin{align}\label{eq-Rayleighquotient}
\cR_{\Omega}(u):=\frac{\int_{\Omega}|\nabla u|^2dv}{\int_{\partial\Omega}u^2dA},
\end{align}
and set $\cR_{\Omega}(u)=+\infty$ when $u\neq 0$ has zero boundary trace. Here $dv$ and $dA$ denote the Riemannian volume form on $\Omega$ and the surface area measure on $\partial\Omega$, respectively. The Steklov eigenvalues satisfy the min-max variational characterization \cite{Ban80}:
\begin{align}\label{eq-steklov-minmax}
\sigma_k(\Omega)=\min_{S\in S_{k+1}}\max_{u\in S\setminus\{0\}}\cR_{\Omega}(u),
\end{align}
where $S_{k+1}$ denotes the family of all $(k+1)$-dimensional subspaces of $H^1(\Omega)$.

The following elementary test function lemma will be useful:
\begin{lem}\label{lem-testfunction}
Let $\Omega\subset \bM^n$ be a bounded Lipschitz domain. Fix $k\in \bN$ and suppose that $u_1,\ldots,u_{k+1}\in H^1(\Omega)$ satisfy $\int_{\partial\Omega}u_i^2dA>0$ for all $i=1,\ldots,k+1$. Assume moreover that, for every $i\neq j$,
\[
\int_{\Omega}\la\nabla u_i,\nabla u_j\ra_gdv=\int_{\partial\Omega}u_iu_jdA=0.
\]
Then we have
\[
\sigma_k(\Omega)\leq \max_{1\leq i\leq k+1}\cR_{\Omega}(u_i).
\]
\end{lem}
\begin{proof}
Let $u=\sum_{i=1}^{k+1}c_iu_i$ be a linear combination of the functions $u_1,\ldots,u_{k+1}$.
We first verify that these functions are linearly independent.
If $u\equiv 0$, then we have
\[
0=\int_{\partial\Omega}u^2dA=\sum_{i=1}^{k+1}c_i^2\int_{\partial\Omega}u_i^2dA.
\]
Since each boundary integral in the sum is positive, it follows that $c_1=\cdots=c_{k+1}=0$, which proves the linear independence. 

Now, we set the $(k+1)$-dimensional subspace to be
\[
S:=\mathrm{span}\{u_1,\ldots,u_{k+1}\}.
\]
Note that
\[
\cR_{\Omega}(u)=\frac{\int_{\Omega}|\sum_{i=1}^{k+1}c_i\nabla u_i|^2dv}{\int_{\partial\Omega}(\sum_{i=1}^{k+1}c_iu_i)^2dA}=\frac{\sum_{i=1}^{k+1}c_i^2\int_{\Omega}|\nabla u_i|^2dv}{\sum_{i=1}^{k+1}c_i^2\int_{\partial\Omega}u_i^2dA}\leq \max_{1\leq i\leq k+1}\cR_{\Omega}(u_i).
\]
The min-max variational characterization \eqref{eq-steklov-minmax} therefore gives
\[
\sigma_k(\Omega)\leq \max_{u\in S\setminus\{0\}}\cR_{\Omega}(u)\leq \max_{1\leq i\leq k+1}\cR_{\Omega}(u_i).
\]
\end{proof}
\begin{rmk}
The orthogonality assumptions hold, in particular, when the functions have pairwise essentially disjoint supports.
\end{rmk}

\subsection{Scaling properties}
Let $\Omega\subset\bR^n$ be a Euclidean domain and let $c>0$. Define
\[
c\Omega:=\{cx:x\in \Omega\}.
\]
For $u\in H^1(\Omega)$, define $u_c\in H^1(c\Omega)$ by 
\[
u_c(x):=u\left(\frac{x}{c}\right),\qquad x\in c\Omega.
\]
A change of variables then gives
\[
\cR_{c\Omega}(u_c)=\frac{\int_{c\Omega}|\nabla u_c|^2dv}{\int_{c\partial\Omega}u_c^2dA}=\frac{c^n\int_{\Omega}c^{-2}|\nabla u|^2dv}{c^{n-1}\int_{\partial\Omega}u^2dA}=c^{-1}\cR_{\Omega}(u).
\]
Consequently, we obtain the basic scaling rule for Steklov eigenvalues:
\[
\sigma_k(c\Omega)=c^{-1}\sigma_k(\Omega).
\]
Moreover,
\[
|c\Omega|=c^n|\Omega|,\qquad D(c\Omega)=cD(\Omega).
\]
Consequently, the two quantities $D(\Omega)\sigma_k(\Omega)$ and $|\Omega|^{\frac{1}{n}}/D(\Omega)$
are scale-invariant.

We next record a comparison principle for uniformly equivalent metrics.
\begin{prop}\label{prop-scalinginvariant}
Let $\Omega$ be a bounded Lipschitz domain in an $n$-dimensional smooth manifold, and let $g_1$ and $g_2$ be Riemannian metrics defined on a neighborhood of $\bar\Omega$. Suppose that, pointwise on $\bar\Omega$,
\[
\Lambda^{-1}g_1\leq g_2\leq \Lambda g_1,
\]
for some $\Lambda\geq 1$. Then for every $k\in \bN$,
\[
\Lambda^{-\frac{2n+1}{2}}\sigma_k(\Omega,g_1)\leq \sigma_k(\Omega,g_2)\leq \Lambda^{\frac{2n+1}{2}}\sigma_k(\Omega,g_1).
\]
\end{prop}
\begin{proof}
Uniform equivalence of the metrics gives
\[
H^1(\Omega,g_1)=H^1(\Omega,g_2)
\]
as sets, with equivalent norms. Moreover, for every $u$ in this common Sobolev space,
\[
\Lambda^{-\frac{n+2}{2}}\int_{\Omega}|\nabla_{g_1}u|_{g_1}^2dv_{g_1}\leq \int_{\Omega}|\nabla_{g_2}u|_{g_2}^2dv_{g_2}\leq \Lambda^{\frac{n+2}{2}}\int_{\Omega}|\nabla_{g_1}u|_{g_1}^2dv_{g_1},
\]
and
\[
\Lambda^{-\frac{n-1}{2}}\int_{\partial\Omega}u^2dA_{g_1}\leq \int_{\partial\Omega}u^2dA_{g_2}\leq \Lambda^{\frac{n-1}{2}}\int_{\partial\Omega}u^2dA_{g_1}.
\]
Combining these estimates yields
\[
\Lambda^{-\frac{2n+1}{2}}\cR_{(\Omega,g_1)}(u)\leq \cR_{(\Omega,g_2)}(u)\leq \Lambda^{\frac{2n+1}{2}}\cR_{(\Omega,g_1)}(u).
\]
Applying the min-max characterization \eqref{eq-steklov-minmax} gives the claimed eigenvalue bounds.
\end{proof}

We use the upper half-space model of hyperbolic space:
\[
\bH^n=\bR^{n-1}\times (0,\infty),\qquad g_{\bH^n}=x_n^{-2}dx^2.
\]
The preceding comparison yields the following approximation result.
\begin{cor}\label{cor-uniformtransform}
For every bounded Lipschitz domain $\Omega\subset \bR^n$, there exists a family of bounded Lipschitz domains $\Omega_t\subset \bH^n$ such that, for every fixed $k\in \bN$,
\[
D(\Omega_t)\sigma_k(\Omega_t)\to D(\Omega)\sigma_k(\Omega),\qquad \frac{|\Omega_t|^{\frac{1}{n}}}{D(\Omega_t)}\to \frac{|\Omega|^{\frac{1}{n}}}{D(\Omega)}
\]
as $t\to 0$.
\end{cor}
\begin{proof}
For sufficiently small $t>0$, let
\[
F_t(x):=e_n+tx,\qquad \Omega_t:=F_t(\Omega). 
\]
Set
\[
g_t:=\frac{1}{(1+tx_n)^2}dx^2.
\]
Then we have
\[
F_t^{\ast}g_{\bH^n}=\frac{t^2}{(1+tx_n)^2}dx^2=t^2g_t.
\]
Since $\Omega$ is bounded, $g_t\to dx^2$ uniformly on $\bar\Omega$. Proposition \ref{prop-scalinginvariant} therefore gives, for every fixed $k\in \bN$,
\[
\sigma_k(\Omega,g_t)\to \sigma_k(\Omega).
\]
The uniform convergence of the metrics also gives $|\Omega|_{g_t}\to |\Omega|$. By invariance under pullback and scaling laws for eigenvalues and volume,
\[
t\sigma_k(\Omega_t)=\sigma_k(\Omega,g_t),\qquad t^{-n}|\Omega_t|=|\Omega|_{g_t}.
\]
On the other hand, the hyperbolic distance formula gives
\[
\cosh d_{\bH^n}(F_t(x),F_t(y))=1+\frac{t^2|x-y|^2}{2(1+tx_n)(1+ty_n)}.
\]
It follows uniformly for $x,y\in \bar\Omega$ that
\[
t^{-1}d_{\bH^n}(F_t(x),F_t(y))\to |x-y|.
\]
Consequently,
\[
t^{-1}D(\Omega_t)\to D(\Omega).
\]
Combining the preceding limits yields the conclusions.
\end{proof}
Consequently, any uniform inequality involving only $D(\Omega)\sigma_k(\Omega)$ and $|\Omega|^{\frac{1}{n}}/D(\Omega)$ that holds for hyperbolic domains must also hold for Euclidean domains.

\subsection{Isoperimetric inequalities}
We recall some basic facts in geometric measure theory. Given a measurable subset $E\subset \bM^n$ and an open set $A\subset \bM^n$, the perimeter of $E$ in $A$ is defined as
\[
P(E;A):=\sup\left\{\int_{E\cap A}\divg_g(X)dv_g:X\in \Gamma_c^{\infty}(TA),\|X\|_{L^{\infty}(A,g)}\leq 1\right\}.
\]
$E$ is said to have finite perimeter in $A$ if $P(E;A)<+\infty$. In particular, if $\Omega$ is a Lipschitz domain, then
\[
P(\Omega;A)=\cH^{n-1}_g(\partial\Omega\cap A).
\]
When $A=\bM^n$, we simply write $P(E)$. The following isoperimetric inequality is well-known:
\begin{thm}\label{thm-isoperimetric}
Let $E\subset \bM^n$ be a measurable set of finite measure and finite perimeter. Then we have
\[
P(E)\gtrsim_n|E|^{\frac{n-1}{n}}.
\]
\end{thm}

We also use the following spectral estimate of Colbois, El Soufi, and Girouard. The estimate was proved for smooth domains in \cite[Theorem 1.2]{CEG11}, and its extension to Lipschitz domains follows from \cite[Proposition 2.24]{CGGS24}.
\begin{thm}\label{thm-CEG11-1.2}
Let $\Omega\subset \bM^n$ be a bounded Lipschitz domain. Then there exists a constant $C(n)>0$ such that for every $k\in \bN$, we have
\begin{align}\label{eq-CEG11-1.3}
P(\Omega)^{\frac{1}{n-1}}\sigma_k(\Omega)\leq C(n)k^{\frac{2}{n}}.
\end{align}
\end{thm}

\section{Sharpness of the exponents}\label{sec-3}
Suppose that, for some $\alpha,\beta\in \bR$, the estimate
\begin{align}\label{eq-alphabeta}
D(\Omega)\sigma_k(\Omega)\leq C(n,\alpha,\beta)k^{\alpha}\left(\frac{|\Omega|^{\frac{1}{n}}}{D(\Omega)}\right)^{\beta}
\end{align}
holds for every bounded Lipschitz domain $\Omega\subset \bM^n$ and every $k\in \bN$. What necessary conditions on $\alpha$ and $\beta$ follow from such an estimate? 

The two examples below address different aspects of sharpness. Thin cylinders determine the largest admissible value of the volume-diameter exponent $\beta$, whereas necklace domains determine the optimal dependence on the eigenvalue index $k$. 

Although the examples below are constructed in $\bR^n$, Corollary \ref{cor-uniformtransform} shows that the same necessary restrictions apply to every uniform hyperbolic estimate of the form \eqref{eq-alphabeta}.

\subsection{Thin cylinders and sharpness}
We first use a family of thin cylinders to show that the exponent $\frac{n}{n-1}$ of the volume-diameter ratio in \eqref{eq-main2} is optimal.

Fix $k\in \bN$ and let $\epsilon>0$ be sufficiently small. Define $\Omega_{\epsilon}=B_{\epsilon}^{n-1}\times (0,1)$, where $B_{\epsilon}^{n-1}\subset \bR^{n-1}$ is an $(n-1)$-dimensional ball of radius $\epsilon$. 
A direct computation gives
\[
|\Omega_{\epsilon}|=\epsilon^{n-1}|B_1^{n-1}|,\qquad D(\Omega_{\epsilon})=\sqrt{1+4\epsilon^2}.
\]
Furthermore, as $\epsilon \to 0^+$, we have
\[
\frac{\sigma_k(\Omega_{\epsilon})}{\epsilon}\to \frac{\pi^2k^2}{n-1}.
\]
The two-dimensional case is proved in \cite[Lemma 3.2]{HM22}, and the same argument applies in higher dimensions.
Fixing $k$ and substituting $\Omega_{\epsilon}$ into \eqref{eq-alphabeta}, we compare the powers of $\epsilon$ as $\epsilon\to 0^+$ to obtain
\[
1\geq \frac{n-1}{n}\beta.
\]

\subsection{Necklace domains and sharpness}
We next use a family of necklace domains to establish the optimal $k$-dependence in our main estimates.

For each $k\in \bN$ and $\epsilon>0$, define $\Omega_{k,\epsilon}$ as the union of $k$ balls of radius $1+\epsilon$ centered at the points $(0,\ldots,0,2i)$ for $i=1,\ldots,k$. For each fixed $k$, as $\epsilon\to 0^+$, we have
\[
|\Omega_{k,\epsilon}|\to k|B_1^n|,\qquad D(\Omega_{k,\epsilon})\to 2k.
\]
Moreover, we have
\[
\lim_{\epsilon\to 0^+}\sigma_k(\Omega_{k,\epsilon})=1.
\]
This convergence is a special case of \cite[Theorem 3.2]{Hon21}; see also \cite{GP10} for the planar case. Hence, for each $k$, we may choose $\epsilon_k>0$ sufficiently small that
\[
D(\Omega_{k,\epsilon_k})\sigma_k(\Omega_{k,\epsilon_k})\geq k,
\]
and
\[
\frac{|\Omega_{k,\epsilon_k}|^{\frac{1}{n}}}{D(\Omega_{k,\epsilon_k})}\asymp_n k^{-\frac{n-1}{n}}.
\]
Substituting $\Omega_{k,\epsilon_k}$ into \eqref{eq-alphabeta} and letting $k\to\infty$, we obtain
\[
1\leq \alpha-\frac{n-1}{n}\beta.
\]

\subsection{Consequences for the sharpness}
The preceding examples yield the following necessary conditions.
\begin{prop}
Assume that there exist constants $\alpha,\beta\in\bR$ and $C>0$ such that the estimate \eqref{eq-alphabeta} holds for every bounded Lipschitz domain $\Omega\subset\bM^n$ and every $k\in\bN$. Then
\[
\beta\leq \frac{n}{n-1},\qquad \text{and}\qquad \alpha\geq 1+\frac{n-1}{n}\beta.
\]
\end{prop}
The proposition has the following consequences:
\begin{itemize}
\item The exponent $\beta=\frac{n}{n-1}$ of the volume-diameter ratio in Theorem \ref{thm-main2} is sharp.
\item Taking $\beta=0$ shows that the linear dependence on $k$ in Theorem \ref{thm-main1} is sharp.
\item Taking $\beta=\frac{n}{n-1}$ shows that the quadratic dependence on $k$ in Theorem \ref{thm-main2} is sharp.
\end{itemize}
However, these examples do not establish the sharpness of the exponent $1+\frac{1}{2n-1}$ in Theorem \ref{thm-main0}; they only show that every uniform estimate with $\beta=0$ must satisfy $\alpha\geq 1$.

\section{The Euclidean diameter estimate}\label{sec-4}
The proof of the Euclidean diameter estimate relies on the annular estimates of \cite{BBG17}. We begin by fixing some notation. For $z\in \bR^n$ and $r>0$, write
\[
B_r(z)=\{x\in \bR^n:|x-z|<r\}.
\]
For brevity, write $B_1:=B_1(0)$. We shall use the relative isoperimetric inequality in balls.
\begin{thm}\label{thm-isoperimetric-ball}
Let $E\subset \bR^n$ be a measurable set of finite measure and finite perimeter. Then for any ball $B_r(z)$, we have
\[
P(E;B_r(z))\gtrsim_n \min\{|E\cap B_r(z)|,|E^c\cap B_r(z)|\}^{\frac{n-1}{n}}.
\]
\end{thm}
This inequality can be found, for example, in \cite[Theorem 5.11]{EG15}.

\subsection{The annular decomposition}

Given $0\leq r_1<r_2$ and $z\in \bR^n$, let $A_{r_1,r_2}(z)$ be the annulus centered at $z$ with inner radius $r_1$ and outer radius $r_2$, namely,
\[
A_{r_1,r_2}(z):=\{x\in \bR^n:r_1<|x-z|<r_2\}.
\]
We shall use the following uniform relative isoperimetric inequality on annuli from \cite{BBG17}. Its proof is based on relative isoperimetric inequalities in balls and in complements of balls.
\begin{lem}[{\cite[Lemma 2.2]{BBG17}}]\label{lem-BBG17-2.2}
Let $m>0$ be given. Then there exists a constant $w=w(m,n)$ such that for every $r\geq 0$ and $l\geq w$, and every measurable set $E\subset A_{r,r+l}(0)$ with $|E|\leq m$, we have
\[
P(E;A_{r,r+l}(0))\gtrsim_n|E|^{\frac{n-1}{n}}.
\]
\end{lem}

Thus, a uniform relative isoperimetric inequality holds on every sufficiently wide annulus. In fact, the constant $w(m,n)$ can be chosen so that 
\begin{align}\label{eq-choicew}
w(m,n)\asymp_nm^{\frac{1}{n}},
\end{align}
which shows that a width of order $m^{\frac{1}{n}}$ is enough. The next lemma is a uniform version of \cite[Lemma 4.1]{BBG17}, stated with the full boundary as in \cite[Remark 4.2]{BBG17}. It is obtained by repeatedly applying Lemma \ref{lem-BBG17-2.2} to successively smaller annuli.
\begin{lem}\label{lem-BBG17-4.1}
Let $m,\lambda>0$ be given. Then there exists a constant $L=L(m,\lambda,n)>w$ such that for every $z\in \bR^n$, $r\geq 0$ and $l\geq L$, and every measurable set $E\subset A_{r,r+l}(z)$ with finite perimeter and $|E|\leq m$, at least one of the following alternatives holds:
\begin{enumerate}[(a)]
\item There exists a function $\phi\in H_0^1(A_{r,r+l}(z))$ with $\int_{\partial E}\phi^2dA>0$ and
\[
\frac{\int_E|\nabla\phi|^2dv}{\int_{\partial E}\phi^2dA}\leq \lambda.
\]
\item We have
\[
\left|E\cap A_{r+\frac{l-w}{2},r+\frac{l+w}{2}}(z)\right|=0,
\]
\end{enumerate}
where $w$ is the constant defined in Lemma \ref{lem-BBG17-2.2}.
\end{lem}

\begin{rmk}
In fact, $L$ can be chosen such that
\begin{align}\label{eq-choiceL}
L-w\asymp_n \lambda^{-\frac{1}{2}}m^{\frac{1}{2n}}.
\end{align}
With $\lambda>0$ fixed, the quantitative choices above imply that 
\[
L(m,\lambda,n)\to 0\qquad \text{as }m\to 0.
\]
\end{rmk}
\begin{rmk}
In \cite[Lemma 4.1]{BBG17}, the result is stated under the assumption $|E|=m$. However, its proof yields a uniform version for $|E|\leq m$. Indeed, the choice $w=w(m,n)$ in Lemma \ref{lem-BBG17-2.2} applies whenever $|E|\leq m$, while the proof of Lemma \ref{lem-BBG17-4.1} only requires 
\[
L-w\gtrsim_n \lambda^{-\frac{1}{2}}|E|^{\frac{1}{2n}}.
\]
Since $|E|\leq m$, the same choice $L=L(m,\lambda,n)$ works.
\end{rmk}

\subsection{Localized test functions}
Let $\Omega\subset \bR^n$ be a bounded Lipschitz domain with diameter $D=D(\Omega)$. After an isometry, we may assume that the diameter is realized by the pair
\[
p=0,\qquad q=De_n
\]
in $\bar\Omega\subset \bR^n$. For each $t\in \bR$, define
\begin{align}\label{eq-levelhyperplane}
H_t:=\{x\in \bR^n:x_n=t\}.
\end{align}
We call $H_t$ the level hyperplane at height $t$. Since $p,q\in \bar\Omega$ and $\Omega$ is connected, we have $\Omega\cap H_t\neq \emptyset$ for every $t\in (0,D)$.

We now combine translations along a prescribed level hyperplane with the annular decomposition of \cite{BBG17}. The resulting construction provides the localized test functions needed in Euclidean space.
\begin{lem}\label{lem-levelhyperplane}
For every $t\in (0,D)$ and every $\rho\in (0,D)$, there exists a point $z\in H_t$ and a function $\phi\in H^1(\bR^n)$ such that
\[
\supp \phi\subset \overline{B_{2\rho}(z)},\qquad \int_{\partial\Omega}\phi^2dA>0.
\]
Moreover,
\[
\frac{\int_{\Omega}|\nabla \phi|^2dx}{\int_{\partial\Omega}\phi^2dA}\lesssim_n \frac{1}{\rho}.
\]
\end{lem}
\begin{proof}
Fix $\lambda=\frac{1}{\rho}$ and let $m=\delta_n\rho^n$, where $\delta_n<\frac{|B_1(0)|}{2}$ is chosen sufficiently small so that the constant $L$ in Lemma \ref{lem-BBG17-4.1} satisfies $L\leq \frac{\rho}{2}$. 

Choose any $x_t\in \Omega\cap H_t$. If $|\Omega\cap B_{\rho}(x_t)|\leq m$, set
\[
E=\Omega\cap A_{\frac{\rho}{4},\frac{3\rho}{4}}(x_t).
\]
Then
\[
|E|\leq |\Omega\cap B_{\rho}(x_t)|\leq m.
\]
Since $\Omega$ is connected, we are in the alternative (a) of Lemma \ref{lem-BBG17-4.1} with $r=\frac{\rho}{4}$ and $l=\frac{\rho}{2}$. Thus, we obtain a function $\phi\in H_0^1(A_{\frac{\rho}{4},\frac{3\rho}{4}}(x_t))$ such that $\int_{\partial\Omega\cap A_{\frac{\rho}{4},\frac{3\rho}{4}}(x_t)}\phi^2dA>0$ and
\[
\frac{\int_{\Omega\cap A_{\frac{\rho}{4},\frac{3\rho}{4}}(x_t)}|\nabla \phi|^2dx}{\int_{\partial\Omega \cap A_{\frac{\rho}{4},\frac{3\rho}{4}}(x_t)}\phi^2dA}\leq \frac{1}{\rho}.
\]
Set $z=x_t$. Extending $\phi$ by zero outside the annulus completes this case.

If $|\Omega\cap B_{\rho}(x_t)|>m$, we may translate the center within $H_t$. By continuity, there exists $z\in H_t$ such that $|\Omega\cap B_{\rho}(z)|=m$. By the relative isoperimetric inequality in Theorem \ref{thm-isoperimetric-ball}, we have
\[
P(\Omega;B_{\rho}(z))\gtrsim_n m^{\frac{n-1}{n}}\gtrsim_n \rho^{n-1}.
\]
Define the function
\[
\phi(x)=\begin{cases}
1,&\text{in }\overline{B_{\rho}(z)},\\
2-\frac{|x-z|}{\rho},&\text{in }A_{\rho,2\rho}(z),\\
0,&\text{in }B_{2\rho}(z)^c.
\end{cases}
\]
Then $\phi\in H^1(\bR^n)$, and
\[
\int_{\Omega}|\nabla \phi|^2dx=\int_{\Omega\cap A_{\rho,2\rho}(z)}|\nabla \phi|^2dx\leq \frac{1}{\rho^2}|A_{\rho,2\rho}(z)|\lesssim_n\rho^{n-2}.
\]
Moreover,
\[
\int_{\partial\Omega}\phi^2dA\geq P(\Omega;B_{\rho}(z))\gtrsim_n\rho^{n-1}.
\]
Combining these two estimates proves the lemma.
\end{proof}

\subsection{The Euclidean diameter bound}
We are now ready to prove Theorem \ref{thm-main1}. 
\begin{proof}
Let $D=D(\Omega)$, and retain the coordinates and notation introduced above. Fix
\[
\rho=\frac{D}{5(k+2)}.
\]
For each $i=1,\ldots,k+1$, set $t_i=\frac{iD}{k+2}$. Applying Lemma \ref{lem-levelhyperplane} at each height $t_i$, we obtain a point $z_i\in H_{t_i}$ and a function $\phi_i$ supported in $\overline{B_{2\rho}(z_i)}$. Moreover,
\[
|z_i-z_j|\geq |t_i-t_j|\geq \frac{D}{k+2}=5\rho>4\rho,\qquad \forall i\neq j.
\]
The supports of the functions $\phi_i$ are therefore pairwise disjoint, and Lemma \ref{lem-testfunction} gives
\[
\sigma_k(\Omega)\leq \max_{1\leq i\leq k+1}\cR_{\Omega}(\phi_i)\lesssim_n\frac{1}{\rho}=\frac{5(k+2)}{D}\lesssim_n \frac{k}{D}.
\]
Multiplying by $D$ completes the proof.
\end{proof}

\section{The level hypersurface method}\label{sec-5}
In this section, we use level hypersurfaces and a one-dimensional reduction to prove Theorem \ref{thm-main2} in both Euclidean and hyperbolic spaces. We then derive Theorem \ref{thm-main0} by interpolation.

\subsection{The level hypersurfaces}
Let $\Omega\subset \bM^n$ be a bounded Lipschitz domain with diameter $D=D(\Omega)$. After an isometry, we may assume that a pair $p,q\in \bar\Omega$ realizing the diameter is given by
\[
p=0,\qquad q=De_n
\]
in the Euclidean case, and by
\[
p=e_n,\qquad q=e^De_n
\]
in the hyperbolic case. Define the height function $h:\bM^n\to \bR$ by
\begin{align}\label{eq-heightfunction}
h(x)=\begin{cases}
x_n,&\bM^n=\bR^n,\\
\log x_n,&\bM^n=\bH^n.
\end{cases}
\end{align}
In both settings,
\[
|\nabla h|=1,
\]
and hence $h$ is $1$-Lipschitz. Together with the diameter condition, this gives
\[
h(p)=0,\qquad h(q)=D,\qquad h(\Omega)\subset (0,D).
\]
For each $t\in \bR$, define
\begin{align}\label{eq-levelhypersurface}
\Sigma_t=\{x\in \bM^n:h(x)=t\}.
\end{align}
We call $\Sigma_t$ the level hypersurface at height $t$. In $\bR^n$, we have $\Sigma_t=H_t$; in $\bH^n$, $\Sigma_t$ is a horosphere. Since $p,q\in \bar\Omega$ and $\Omega$ is connected, we have $\Omega\cap \Sigma_t\neq \emptyset$ for every $t\in (0,D)$.

\subsection{Reduction to one dimension}
We now construct test functions that depend only on the height $h$.

Recalling the level hypersurfaces $\Sigma_t$ defined in \eqref{eq-levelhypersurface}, we define the slice volume function by
\begin{align}\label{eq-levelarea}
a(t)=\cH^{n-1}_g(\Omega\cap \Sigma_t).
\end{align}
Since $\Omega\cap \Sigma_t$ is a nonempty relatively open subset of $\Sigma_t$ for every $t\in (0,D)$, we have
\[
a(t)>0,\qquad t\in (0,D).
\]
Moreover, since $|\nabla h|=1$, the coarea formula gives
\[
\int_0^Da(t)dt=|\Omega|.
\] 
The following proposition reduces the Rayleigh-Steklov quotients of height-dependent test functions to weighted one-dimensional quotients.
\begin{prop}\label{prop-rayleighreduction}
Let $I\subset (0,D)$ be an interval, and let $f\in H_0^1(I)\setminus\{0\}$. Extend $f$ by zero to $(0,D)$, and define
\[
u:=f\circ h.
\]
Then 
\[
\cR_{\Omega}(u)\lesssim_n\frac{\int_Ia(t)|f'(t)|^2dt}{\int_Ia(t)^{\frac{n-2}{n-1}}f(t)^2dt}.
\]
\end{prop}
\begin{proof}
By the chain rule and the identity $|\nabla h|=1$, we have $|\nabla u|^2=|f'(h)|^2$. Thus, the coarea formula gives
\begin{align}\label{eq-onedimest1}
\int_{\Omega}|\nabla u|^2dv=\int_Ia(t)|f'(t)|^2dt.
\end{align}
On the other hand, the coarea formula on $\partial\Omega$ gives
\[
\int_{\partial\Omega}u^2|\nabla_{\partial\Omega}h|dA=\int_IP_{\Sigma_t}(\Omega\cap \Sigma_t)f(t)^2dt,
\]
where $P_{\Sigma_t}(\Omega\cap \Sigma_t)$ denotes the perimeter computed in $\Sigma_t$. Since every $\Sigma_t$ is isometric to $\bR^{n-1}$ with its Euclidean metric, the isoperimetric inequality from Theorem \ref{thm-isoperimetric} gives
\[
P_{\Sigma_t}(\Omega\cap \Sigma_t)\gtrsim_na(t)^{\frac{n-2}{n-1}}.
\]
Since $|\nabla_{\partial\Omega}h|\leq |\nabla h|=1$, we have
\begin{align}\label{eq-onedimest2}
\int_{\partial\Omega}u^2dA\geq \int_{\partial\Omega}u^2|\nabla_{\partial\Omega}h|dA=\int_IP_{\Sigma_t}(\Omega\cap \Sigma_t)f(t)^2dt\gtrsim_n\int_Ia(t)^{\frac{n-2}{n-1}}f(t)^2dt.
\end{align}
Dividing the identity \eqref{eq-onedimest1} by the lower bound in \eqref{eq-onedimest2} proves the proposition.
\end{proof}

The remaining difficulty is that the slice volume function $a$ may be highly irregular, so the test function $f$ must be chosen carefully. The following lemma addresses this issue by selecting a suitable dyadic level set of $a$ and deriving the required weighted estimate.
\begin{lem}\label{lem-1d-reduction}
Let $0\leq \gamma<1$. Then there exists a constant $C(\gamma)>0$ such that, for every positive function $a\in L^1(0,1)$ satisfying
\[
\int_0^1a(t)dt=1,
\]
there exists $f\in H_0^1(0,1)\setminus\{0\}$ such that
\[
\frac{\int_0^1a(t)|f'(t)|^2dt}{\int_0^1a(t)^{\gamma}f(t)^2dt}\leq C(\gamma).
\]
\end{lem}
\begin{proof}
Define the level sets
\[
E_j:=\{t\in (0,1):2^{-j}<a(t)\leq 2^{-j+1}\},\qquad j=0,1,2,\ldots.
\]
Markov's inequality yields
\[
\sum_{j=0}^{\infty}|E_j|=|\{0<a(t)\leq 2\}|=1-|\{a(t)>2\}|\geq 1-\frac{1}{2}=\frac{1}{2}.
\]
Hence there exists an integer $j\geq 0$ such that
\[
|E_j|\geq \frac{1}{2}(1-2^{-\frac{1-\gamma}{2}})2^{-\frac{1-\gamma}{2}j}.
\]
Fix such an index $j$ and set $E:=E_j$. Let $s(t):=\int_0^t\chi_E(r)dr$ be the cumulative distribution function of $E$, and let $T(r):=2\min\{r,1-r\}$. Define the function
\[
f(t)=T\left(\frac{s(t)}{|E|}\right).
\]
By the chain rule for Lipschitz compositions, we have $f\in H_0^1(0,1)\setminus\{0\}$, $f(0)=f(1)=0$ and $|f'(t)|=\frac{2}{|E|}\chi_E(t)$ almost everywhere. Thus we have
\[
\int_0^1a(t)|f'(t)|^2dt=\frac{4}{|E|^2}\int_Ea(t)dt\leq \frac{4}{|E|^2}2^{-j+1}|E|=\frac{2^{-j+3}}{|E|},
\]
where we used the definition of $E_j$. Since $s$ is absolutely continuous and nondecreasing, the one-dimensional area formula yields
\[\begin{aligned}
\int_Ef(t)^2dt=&\int_0^1T\left(\frac{s(t)}{|E|}\right)^2\chi_E(t)dt\\
=&\int_0^1T\left(\frac{s(t)}{|E|}\right)^2s'(t)dt\\
=&\int_0^{|E|}T\left(\frac{y}{|E|}\right)^2dy\\
=&|E|\int_0^1T(r)^2dr.
\end{aligned}\]
Therefore,
\[
\int_0^1a(t)^{\gamma}f(t)^2dt\geq 2^{-\gamma j}\int_Ef(t)^2dt= 2^{-\gamma j}|E|\int_0^1T(r)^2dr=2^{-\gamma j}\frac{|E|}{3}.
\]
Combining these two estimates, we conclude
\[
\frac{\int_0^1a(t)|f'(t)|^2dt}{\int_0^1a(t)^{\gamma}f(t)^2dt}\leq \frac{3\cdot 2^{-j+3}}{2^{-\gamma j}|E|^2}\leq \frac{3\cdot 2^{-j+5}}{(1-2^{-\frac{1-\gamma}{2}})^22^{-\gamma j}2^{-(1-\gamma)j}}=\frac{96}{(1-2^{-\frac{1-\gamma}{2}})^2}.
\]
\end{proof}

\begin{rmk}
The restriction $\gamma<1$ is essential: the conclusion of Lemma \ref{lem-1d-reduction} fails at the endpoint $\gamma=1$. Indeed, for $M>0$, define
\[
a_M(t)=\frac{Me^{-Mt}}{1-e^{-M}}.
\]
Then $\int_0^1a_M(t)dt=1$.
For any $f\in H_0^1(0,1)\setminus\{0\}$, set
\[
g(t)=e^{-\frac{Mt}{2}}f(t).
\]
Then $g\in H_0^1(0,1)\setminus\{0\}$. Integration by parts gives
\[
\frac{\int_0^1a_M(t)|f'(t)|^2dt}{\int_0^1a_M(t)f(t)^2dt}=\frac{\int_0^1|g'(t)|^2dt}{\int_0^1g(t)^2dt}+\frac{M^2}{4}\geq \pi^2+\frac{M^2}{4}.
\]
Thus, at $\gamma=1$, no uniform constant independent of $a$ can exist.
\end{rmk}

Let $I=(\alpha,\alpha+l)$ be a bounded interval with $l>0$, and let $a\in L^1(I)$ be positive. Consider
\[
A(r)=\frac{l}{\int_Ia(t)dt}a(\alpha+lr),\qquad r\in (0,1).
\]
Fix $0\leq \gamma<1$. Since
\[
\int_0^1A(r)dr=1,
\]
Lemma \ref{lem-1d-reduction} yields a function $g\in H_0^1(0,1)\setminus\{0\}$ such that
\[
\frac{\int_0^1A(r)|g'(r)|^2dr}{\int_0^1A(r)^{\gamma}g(r)^2dr}\leq C(\gamma).
\]
Define
\[
f(t)=g\left(\frac{t-\alpha}{l}\right).
\]
A change of variables then yields the following corollary.
\begin{cor}\label{cor-1d-reduction}
Fix $0\leq \gamma<1$. Let $I\subset \bR$ be a bounded interval of positive length, and let $a\in L^1(I)$ be positive. Then there exists $f\in H_0^1(I)\setminus\{0\}$ such that
\[
\frac{\int_Ia(t)|f'(t)|^2dt}{\int_Ia(t)^{\gamma}f(t)^2dt}\leq C(\gamma)\left(\int_Ia(t)dt\right)^{1-\gamma}|I|^{-3+\gamma}.
\]
\end{cor}

Combining Corollary \ref{cor-1d-reduction} with Proposition \ref{prop-rayleighreduction}, we obtain height-dependent test functions whose Rayleigh-Steklov quotients are controlled in terms of the length of the interval and the volume of the corresponding slab.

\subsection{Proof of the volume-diameter estimate}
We now prove Theorem \ref{thm-main2}. The following selection argument is inspired by the capacitary method of Grigor'yan, Netrusov, and Yau \cite[Section 4]{GNY04}. We partition the height range into $2k$ intervals and retain $k+1$ intervals whose associated slabs have small volume. On each retained interval, Corollary \ref{cor-1d-reduction}, applied with $\gamma=\frac{n-2}{n-1}$, provides a one-dimensional test function; we then extend these functions by zero to the full height interval and lift them to test functions on $\Omega$ via Proposition \ref{prop-rayleighreduction}.

\begin{proof}
Let $D=D(\Omega)$ and retain the height function and slice-volume notation introduced above. Fix $k\in \bN$ and divide $(0,D)$ into $2k$ intervals $I_1,\ldots,I_{2k}$ of equal length $\frac{D}{2k}$. Define the corresponding slab volumes by
\[
a_i:=\int_{I_i}a(t)dt,\qquad i=1,\ldots,2k.
\]
Since the slab volumes $a_i$ sum to $|\Omega|$, at most $k-1$ of the $2k$ intervals can satisfy $a_i>\frac{|\Omega|}{k}.$
Consequently, at least $k+1$ intervals satisfy 
\begin{align}\label{eq-choiceai}
a_i\leq\frac{|\Omega|}{k}.
\end{align}
Relabel $k+1$ of these intervals, together with their corresponding slab volumes, as $I_1,\ldots,I_{k+1}$ and $a_1,\ldots,a_{k+1}$. For each $i=1,\ldots,k+1$, applying Corollary \ref{cor-1d-reduction} with $\gamma=\frac{n-2}{n-1}$ yields a function $f_i\in H_0^1(I_i)\setminus\{0\}$ such that
\begin{align}\label{eq-estimateforonedim}
\frac{\int_{I_i}a(t)|f_i'(t)|^2dt}{\int_{I_i}a(t)^{\frac{n-2}{n-1}}f_i(t)^2dt}\lesssim_na_i^{\frac{1}{n-1}}|I_i|^{-2-\frac{1}{n-1}}.
\end{align}
Using \eqref{eq-choiceai} and the identity 
\[
|I_i|=\frac{D}{2k},
\]
we can bound the right-hand side, up to a dimensional constant, by
\[
\left(\frac{|\Omega|}{k}\right)^{\frac{1}{n-1}}\left(\frac{D}{2k}\right)^{-2-\frac{1}{n-1}}\asymp_nk^2|\Omega|^{\frac{1}{n-1}}D^{-2-\frac{1}{n-1}}.
\]
Extend each $f_i$ by zero to $(0,D)$, and define $u_i:=f_i\circ h$. We thus obtain $k+1$ functions $u_1,\ldots,u_{k+1}$ in $H^1(\Omega)$ with pairwise essentially disjoint supports. By Proposition \ref{prop-rayleighreduction} and \eqref{eq-estimateforonedim}, we have
\[
\cR_{\Omega}(u_i)\lesssim_n k^2|\Omega|^{\frac{1}{n-1}}D^{-2-\frac{1}{n-1}},\qquad \forall i=1,\ldots,k+1.
\]
Applying Lemma \ref{lem-testfunction}, we conclude
\begin{align}\label{eq-mainthm2eq}
\sigma_k(\Omega)\lesssim_n k^2|\Omega|^{\frac{1}{n-1}}D^{-2-\frac{1}{n-1}},
\end{align}
which is equivalent to the estimate in Theorem \ref{thm-main2}.
\end{proof}

\subsection{The hyperbolic diameter estimate}
Finally, we prove Theorem \ref{thm-main0} by interpolating the hyperbolic case of Theorem \ref{thm-main2} with the volume-normalized estimate obtained from Theorem \ref{thm-CEG11-1.2} and the hyperbolic isoperimetric inequality. This interpolation argument is inspired by \cite{Pro26}.
\begin{proof}
By \eqref{eq-CEG11-1.3} and the isoperimetric inequality in Theorem \ref{thm-isoperimetric}, we obtain
\[
\sigma_k(\Omega)\lesssim_n|\Omega|^{-\frac{1}{n}}k^{\frac{2}{n}}.
\]
Interpolating this estimate with \eqref{eq-mainthm2eq} to eliminate the volume term, we obtain
\[\begin{aligned}
\sigma_k(\Omega)=&\sigma_k(\Omega)^{\frac{n}{2n-1}}\sigma_k(\Omega)^{\frac{n-1}{2n-1}}\\
\lesssim_n&\left(|\Omega|^{-\frac{1}{n}}k^{\frac{2}{n}}\right)^{\frac{n}{2n-1}}\left(k^2|\Omega|^{\frac{1}{n-1}}D^{-2-\frac{1}{n-1}}\right)^{\frac{n-1}{2n-1}}\\
=&k^{\frac{2n}{2n-1}}D^{-1},
\end{aligned}\]
which is equivalent to the estimate in Theorem \ref{thm-main0}.
\end{proof}

\end{document}